\documentclass[a4paper,11pt,reqno]{amsart}
\usepackage[usenames,dvipsnames,svgnames,table]{xcolor}
\usepackage[utf8]{inputenc}
\usepackage[T1]{fontenc}
\usepackage{lmodern,microtype,mathtools,amssymb}
\usepackage[foot]{amsaddr}
\usepackage{hyperref,cleveref}
\usepackage[margin=1in,marginparwidth=.8in,includehead]{geometry}
\usepackage{paralist}
\usepackage{comment}

\hypersetup{
  pdfauthor={Vida Dujmović, Gwenaël Joret, Piotr Micek, Pat Morin},
  colorlinks,
  linkcolor={red!50!black},
  citecolor={blue!50!black},
  urlcolor={blue!80!black}
}

\newtheorem{theorem}{Theorem}
\newtheorem{lemma}[theorem]{Lemma}

\newtheorem{corollary}[theorem]{Corollary}
\theoremstyle{remark}

\newcommand\Oh{\mathcal{O}}

\DeclareMathOperator\pw{pw}

\DeclarePairedDelimiter\set{\{}{\}}

\let\leq\leqslant
\let\geq\geqslant
\let\le\leqslant
\let\ge\geqslant

\let\subset\subseteq

\let\setminus-
\let\Gamma\varGamma
\let\Lambda\varLambda
\let\Phi\varPhi
\let\Psi\varPsi
\let\rho\varrho

\let\epsilon\varepsilon

\makeatletter
\let\old@setaddresses\@setaddresses
\def\@setaddresses{\bigskip\bgroup\parindent 0pt\let\scshape\relax\old@setaddresses\egroup}
\makeatother

\begin{document}

\title{Tight bound for the Erd\H{o}s-Pósa property of tree minors}

\author[Dujmović]{Vida Dujmović}
\email{vida.dujmovic@uottawa.ca}
\author[Joret]{Gwenaël Joret}
\email{gwenael.joret@ulb.be}
\author[Micek]{Piotr Micek}
\email{piotr.micek@uj.edu.pl}
\author[Morin]{Pat Morin}
\email{morin@scs.carleton.ca}

\address[V.~Dujmovi{\'c}]{School of Computer Science and Electrical Engineering, University of Ottawa, Ottawa, Canada}
\address[G.~Joret]{Computer Science Department, Université libre de Bruxelles, Brussels, Belgium}
\address[P.~Micek]
{Theoretical Computer Science Department, Faculty of Mathematics and Computer Science,
Jagiellonian University, Kraków, Poland}
\address[P.~Morin]{School of Computer Science, Carleton University, Ottawa, Canada}
\thanks{G.~Joret is supported by a PDR grant from the Belgian National Fund for Scientific Research (FNRS). 
V.~Dujmović is supported by NSERC and a University of Ottawa Research Chair.
P.~Micek is supported by the National Science Center of Poland under grant
UMO-2023/05/Y/ST6/00079 within the WEAVE-UNISONO program. P.~Morin is supported by NSERC}

\begin{abstract}
Let $T$ be a tree on $t$ vertices. 
We prove that for every positive integer $k$ and every graph $G$, 
either $G$ contains $k$ pairwise vertex-disjoint subgraphs each having a $T$ minor, or 
there exists a set $X$ of at most $t(k-1)$ vertices of $G$ such that  
$G-X$ has no $T$ minor. 
The bound on the size of $X$ is best possible and improves on an earlier $f(t)k$ bound proved by Fiorini, Joret, and Wood (2013) with some fast growing function $f(t)$. 
Moreover, our proof is short and simple. 
\end{abstract}

\maketitle

\section{Introduction}

In 1965, Erd\H{o}s and P\'osa~\cite{EP1965} showed that every graph $G$ either contains $k$ vertex-disjoint cycles or contains a set $X$ of $\Oh(k \log k)$ vertices such that $G-X$ has no cycles. 
The $\Oh(k \log k)$ bound on the size of $X$ is best possible up to a constant factor. 
Using their Grid Minor Theorem, Robertson and Seymour~\cite{RS1986} proved the following generalization: For every planar graph $H$, there exists a function $f_H(k)$ such that every graph $G$ contains either $k$ vertex-disjoint subgraphs each having an $H$ minor, or a set $X$ of at most $f_H(k)$ vertices such that $G-X$ has no $H$ minor. 
For $H=K_3$, this corresponds to the setting of the Erd\H{o}s-P\'osa theorem.  

The theorem of Robertson and Seymour is best possible in the sense that no such result holds when $H$ is not planar. 
The original upper bound of $f_H(k)$ on the size of $X$ depends on bounds from the Grid Minor Theorem and is large as a result (though it is polynomial in $k$ if we use the polynomial version of the Grid Minor Theorem, see~\cite{CC2016}). 
Chekuri and Chuzhoy~\cite{CC2013} subsequently showed an improved upper bound of $\Oh_H(k \log^c k)$ for a fixed planar graph $H$, where $c$ is some large but absolute constant. 
This was in turn improved to $\Oh_H(k \log k)$ by Cames van Batenburg, Huynh, Joret, and Raymond~\cite{CvBHJR2019}, thus matching the original bound of Erd\H{o}s and P\'osa for cycles. 

An $\Oh_H(k \log k)$ bound is best possible when $H$ contains a cycle. 
However, when $H$ is a forest, it turns out that one can obtain a linear in $k$ bound on the size of $X$, as proved by Fiorini, Joret, and Wood~\cite{FJW2013}. 
Their proof gives an $\Oh_H(k)$ bound with a non-explicit constant factor that grows very fast as a function of $|V(H)|$. 
This is due to the use of MSO-based tools in the proof, among others.
In this short note, we give a simple proof of their result with an optimal dependence on $t$ and $k$ when $H$ is a tree.

\begin{theorem}
\label{thm:main-tree}
Let $T$ be a tree on $t$ vertices. 
For every positive integer $k$ and every graph $G$, 
either $G$ contains $k$ pairwise vertex-disjoint subgraphs each having a $T$ minor, or 
there exists a set $X$ of at most $t(k-1)$ vertices of $G$ such that  
$G-X$ has no $T$ minor. 
\end{theorem}
\noindent Observe that the bound on the size of $X$ in~\Cref{thm:main-tree} is tight: If $G$ is a complete graph on $tk-1$ vertices, then $G$ does not contain $k$ pairwise vertex-disjoint subgraphs each having a $T$ minor, and every set $X$ of vertices such that $G-X$ has no $T$ minor has size at least $|V(G)| - (t-1) = t(k-1)$.

\Cref{thm:main-tree} follows immediately from the following more general result for forests.

\begin{theorem}
\label{thm:main-forest}
Let $F$ be a forest on $t$ vertices and let $t'$ be the maximum number of vertices in a component of $F$. 
For every positive integer $k$ and every graph $G$, 
either $G$ contains $k$ pairwise vertex-disjoint subgraphs each having an $F$ minor, or 
there exists a set $X$ of at most $tk-t'$ vertices of $G$ such that  
$G-X$ has no $F$ minor. 
\end{theorem}




Let us also point out the following corollary of~\Cref{thm:main-tree} (proved in the next section).  

\begin{corollary}
\label{cor:pathwidth}
For all positive integers $p$ and $k$, and for every graph $G$, either $G$ contains $k$ vertex-disjoint subgraphs each of pathwidth at least $p$, or $G$ contains a set $X$ of at most $2\cdot 3^{p+1}k$ vertices such that $G-X$ has pathwidth strictly less than $p$. 
\end{corollary}

\section{Proof}
For a positive integer $k$, we use the notation $[k]\coloneqq \{1,\ldots,k\}$, and when $k=0$ let $[k] \coloneqq \varnothing$. 

Let $G$ be a graph.
We denote by $V(G)$ and $E(G)$, the vertex set and edge set of $G$, respectively. 
Let $X\subset V(G)$.  Then $G[X]$ denotes the subgraph of $G$ induced by the vertices in $X$ and $G-X=G[V(G)\setminus X]$. 
We define the \emph{boundary} of $X$ in $G$ to be $\partial_G X := \set{v \in X \mid vw\in E(G),\, w\in V(G-X)}$.  We omit the subscript $G$ when the graph $G$ is clear from the context. 

A \emph{path decomposition} of $G$ is a sequence $(B_1, B_2, \dots, B_q)$ of vertex subsets of $G$ called {\em bags} satisfying the following properties: (1) every vertex of $G$ appears in a nonempty set of consecutive bags, and (2) for every edge $uv$ of $G$, there is a bag containing both $u$ and $v$. 
The {\em width} of the path decomposition is the maximum size of a bag minus $1$. 
The {\em pathwidth $\pw(G)$} of $G$ is the minimum width of a path decomposition of $G$.

A graph $H$ is a {\em minor} of a graph $G$ if $H$ can be obtained from a subgraph of $G$ by contracting edges. 
Robertson and Seymour~\cite{RS83} proved that there exists a function $f:\mathbb{N}\to\mathbb{N}$ such that
for every graph $G$ and every forest $F$ on $t$ vertices, 
if $\pw(G)\geq f(t)$ then $G$ contains $F$ as a minor. 
Bienstock, Robertson, Seymour, and Thomas~\cite{BRST1991} later showed that one can take $f(t)=t-1$, which is best possible. 
Diestel~\cite{D95} subsequently gave a short proof of this result. 
Our proof of~\Cref{thm:main-forest} builds on the following slightly stronger result, which appears implicitly in Diestel's proof~\cite{D95}. 

\begin{lemma}[\cite{D95}]
\label{lem:Diestels-thing}
Let $G$ be a graph, let $t$ be a positive integer, and let $F$ be a forest on $t$ vertices. 
If $\pw(G)\geq t-1$, then there exists $Y\subseteq V(G)$ such that 
\begin{enumerate}
\item $G[Y]$ has a path decomposition $(B_1,\ldots,B_q)$ of width at most $t-1$ such that $\partial Y\subseteq B_q$, and
\item $G[Y]$ contains $F$ as a minor.
\end{enumerate}
\end{lemma}

We now turn to the proof of~\Cref{thm:main-forest}. 

\begin{proof}[Proof of~\Cref{thm:main-forest}]
We prove the following strengthening of \Cref{thm:main-forest}:
Let $G$ be a graph, 
let $c$ be a positive integer, 
let $t_1\le \cdots\le t_c$ be nonnegative integers,
let $T_1,\ldots,T_c$ be trees with $|V(T_i)|=t_i$ for every $i\in[c]$, 
let $x_1,\ldots,x_c$ be nonnegative integers, at least one of which is nonzero, and let 
$I:= \{i\in [c]\mid x_i \geq 1\}$. 
Then either 
\begin{enumerate}
\item $G$ contains 
pairwise vertex-disjoint subgraphs $\set{M_{i, j}\mid i\in[c],\, j\in[x_i]}$ such that, for each $i\in [c]$ and $j\in[x_i]$, $M_{i,j}$ contains a $T_i$ minor, or 

\item \label{item:second}
there exists $X\subseteq V(G)$ with $|X|\leq \sum_{i\in I}x_it_i - t_{\max(I)}$ and $G-X$ does not contain $T_i$ as a minor for some $i\in I$. 
\end{enumerate}
We call the tuple $(G,c,T_1,\ldots,T_c,x_1,\ldots,x_c)$ an \emph{instance}. 
\Cref{thm:main-forest} follows by letting $T_1,\ldots,T_c$ be the components of the forest $F$ and letting $x_1=x_2=\cdots=x_c=k$. 

Roughly, the proof describes an inductive procedure that attempts to find a pairwise disjoint collection of models, where the number of models of each tree $T_i$ is $x_i$.  
Induction is on the number $\sum_{i\in [c]} x_i$ of models still missing from the collection.  Failing to find one of the missing models at some step will establish \eqref{item:second}.

Let $(G,c,T_1,\ldots,T_c,x_1\ldots,x_c)$ be an instance, and let $m:=\min(I)$. 
Then $T_m$ is a smallest tree among $T_1,\ldots,T_c$ such that $x_m \geq 1$, that is, such that we are still missing a model of $T_m$. 
In the base case, $\sum_{i\in [c]} x_i=x_{m}=1$, and either $G$ has a $T_{m}$ minor and the first outcome of the statement holds, or $G$ has no such minor and the second outcome holds with $X:=\varnothing$, since 
$\sum_{i\in I} x_it_i -t_{\max(I)}= t_m - t_m = 0$.

For the inductive case, assume that $\sum_{i\in [c]} x_i \geq 2$, and that the statement holds for instances with smaller values of the sum. 
If, for every $i\in I$, $G$ has no $T_i$ minor then the second outcome of the statement holds with $X:=\varnothing$ again. 
Thus, we may assume that $G$ has a $T_i$ minor for some $i\in I$.

If $G$ has pathwidth at least $t_{m}-1$, apply~\Cref{lem:Diestels-thing} with $t=t_{m}$ and $F=T_{m}$, and let $Y$ be the resulting subset of vertices of $G$. 
If $G$ has pathwidth less than $t_{m}-1$, simply let $Y:=V(G)$.   In either case, $G[Y]$ has pathwidth at most $t_m-1$ and has a path decomposition $(B_1, B_2, \dots, B_q)$ with $|B_\ell|\leq t_{m}$ for all $\ell\in[q]$, and such that $\partial_G  Y \subseteq B_q$. See \Cref{y_gl_bl}.
Furthermore, observe that in both cases $G[Y]$ has a $T_i$ minor for some $i\in I$, by our assumption on $G$. 

\begin{figure}[htpb]
    \begin{center}
        \includegraphics{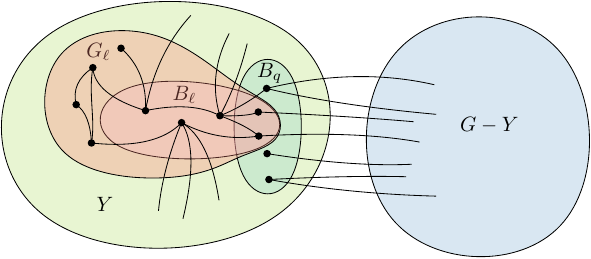}
    \end{center}
    \caption{The set $Y$ and the graph $G_\ell$ whose boundary in $G$ is contained in $B_\ell$.}
    \label{y_gl_bl}
\end{figure}

Let $\ell \in [q]$ be the smallest index such that $G_\ell:=G[B_1 \cup \cdots \cup B_\ell]$ contains a $T_{i}$ minor for some $i\in I$, and let $i'$ be an index in $I$ such that 
\begin{equation}
\label{my_eq-first}
\textrm{$G_\ell$ contains a $T_{i'}$ minor.}
\tag{$\star$} 
\end{equation}

Observe that 
\begin{equation}
\label{my_eq}
\textrm{$G_\ell-B_\ell$ has no $T_i$ minor for every $i\in I$.}
\tag{$\star\star$} 
\end{equation}

We claim that 
\begin{equation}
\label{triple_stars}
\textrm{there is no edge in $G$ between vertices of  $G_\ell-B_\ell$ and vertices of $G-V(G_\ell)$.}
\tag{$\star\star$$\star$} 
\end{equation} 
To see this, suppose for a contradiction that $uv$ is such an edge, with $u\in V(G_\ell)-B_\ell$ and $v\in V(G)-V(G_\ell)$. 
First, note that $u\in B_1 \cup \cdots \cup B_{\ell-1}$. 
If $v\in Y$, then $u$ and $v$ appear together in some bag $B_j$ of the path decomposition $(B_1, B_2, \dots, B_q)$ of $G[Y]$, and $j > \ell$ since $v \notin B_1 \cup \cdots \cup B_\ell$. 
However, since 
$u\in B_1 \cup \cdots \cup B_{\ell-1}$ and $u\in B_j$, we conclude that $u$ belongs also to $B_\ell$, a contradiction. 
If $v\notin Y$, then $u\in \partial Y$, and thus $u\in B_q$. 
Again, we deduce similarly that $u \in B_\ell$, a contradiction. 
This completes the proof of \eqref{triple_stars}. 

Let $G':=G - V(G_\ell)$. 
Let $x'_i := x_i$ for each $i\in [c]-\{i'\}$ and let 
$x'_{i'} := x_{i'}-1$. 
Let $I'=\set{i\in[c]\mid x'_i\geq1}$.
Apply induction to the instance $(G',c,T_1,\ldots,T_c,x'_1,\ldots,x'_c)$. 
If it results in a set of vertex-disjoint subgraphs $\{M'_{i,j}\mid i\in[c],\, j\in [x'_i]\}$, with $M'_{i,j}$ containing a $T_i$ minor for each $i\in [c]$ and $j \in [x'_i]$, then we let $M_{i,j}:=M'_{i,j}$ for each $i\in [c]$ and $j \in [x'_i]$, and $M_{i',x_{i'}} := G_\ell$, 
which using~\eqref{my_eq-first} results in the desired collection of vertex-disjoint subgraphs. 
Otherwise, we obtain a set $X'$ of at most $\sum_{i\in I'}x'_it_i -t_{\max(I')}$ vertices such that $G'-X'$ does not contain $T_a$ as a minor for some $a\in I'$. 

Let $X:=X'\cup B_\ell$. 
Observe that 
\begin{align*}
|X| = |X'|+|B_{\ell}|
&\leq \sum_{i\in I'}x'_it_i - t_{\max(I')}  + t_{m} 
\leq \sum_{i\in I}x_it_i - (t_{\max(I')} + t_{i'}  - t_{m})\\
&\leq \sum_{i\in I}x_it_i - t_{\max(I)}.
\end{align*}
To see why the last inequality holds, there are two cases to consider:
\begin{inparaenum}[(i)]
    \item If $\max(I')=\max(I)$ then the inequality follows immediately since $t_{i'}\ge t_m$.
    \item If $\max(I')<\max(I)$ then $i'=\max(I)$ and $\max(I')\ge \min(I')=m$, so $t_{\max(I')}+t_{i'}-t_m \ge t_{i'}=t_{\max(I)}$.
\end{inparaenum}

Now, let us show that $G-X$ does not contain $T_i$ as a minor, for some $i\in I$.  
Let $a\in I'$ be such that $G'-X'$ does not contain $T_a$ as minor. 
We will show that we can take $i=a$.  
To do so, it is enough to show that $X$ meets every inclusion-wise minimal subgraph of $G$ containing a $T_a$ minor. 
Let $M$ be such a subgraph of $G$. 
Note that $M$ is connected, since $T_a$ is connected.  
Now, observe that by~\eqref{triple_stars}, either $M$ 
is contained in $G'$, or $M$ is contained in $G_\ell-B_\ell$, or $M$ contains a vertex of $B_\ell$.  In the first case, $M$ contains a vertex of $X'\subseteq X$, by the choice of $a$. 
The second case is ruled out by~\eqref{my_eq}. 
In the third case, $M$ contains a vertex of $B_\ell\subseteq X$.  Thus, 
we conclude that $M$ contains a vertex of $X$. 
This concludes the proof. 
\end{proof}

We may now turn to the proof of~\Cref{cor:pathwidth}.  We will use the following lemma, which is a special case of a more general result of Robertson and Seymour [Statement (8.7) in~\cite{RS1986}].

\begin{lemma}\label{lemma:helly_property_tree_decomposition}
    For every graph $G$, for every path decomposition $(B_1, B_2, \dots, B_q)$ of $G$, for every family $\mathcal{F}$ of connected subgraphs of $G$, for every positive integer $d$, either:
    \begin{enumerate}
        \item there are $d$ pairwise vertex-disjoint subgraphs in $\mathcal{F}$, or
        \item there is a set $X$ that is the union of at most $d-1$ bags of $(B_1, B_2, \dots, B_q)$ such that $V(F) \cap X \neq \varnothing$ for every $F \in \mathcal{F}$.
    \end{enumerate}
\end{lemma}

\begin{proof}[Proof of~\Cref{cor:pathwidth}]
It is known (and an easy exercise to show) that, for every positive integer $p$, 
the complete ternary tree $T_p$ of height $p$ has pathwidth $p$. 
First, apply~\Cref{thm:main-tree} on $G$ with the tree $T_p$. 
If $G$ contains $k$ vertex-disjoint subgraphs each containing a $T_p$ minor, we are done. 
So we may assume that the theorem produces a set $X_1$ of at most $|V(T_p)|(k-1) \leq 3^{p+1}(k-1)$ vertices such that $G-X_1$ has no $T_p$ minor. 

By~\Cref{lem:Diestels-thing}, $G-X_1$ has a path decomposition $(B_1, B_2, \dots, B_q)$ of width strictly less than $3^{p+1}$. 
It is easily checked that every inclusion-wise minimal subgraph of $G-X_1$ with pathwidth at least $p$ is connected. Apply~\Cref{lemma:helly_property_tree_decomposition} on $G-X_1$ with the path decomposition $(B_1, B_2, \dots, B_q)$, with $d=k$, and with the family $\mathcal{F}$ of connected subgraphs of $G-X_1$ with pathwidth at least $p$. 
If $\mathcal{F}$ contains $k$ pairwise vertex-disjoint members, we are done. 
So we may assume that the lemma produces a set $X_2$ of at most $3^{p+1}(k-1)$ vertices such that $X_2$ hits every member of $\mathcal{F}$. 
It follows that $G-X_1-X_2$ has pathwidth strictly less than $p$. 
Let $X\coloneqq X_1 \cup X_2$. 
Since $|X| \leq 3^{p+1}(k-1) + 3^{p+1}(k-1) \leq 2\cdot 3^{p+1}k$, the set $X$ has the desired properties. 
\end{proof}

\section*{Acknowledgments}
This work was done during a visit of Gwenaël Joret and Piotr Micek to the University of Ottawa and Carleton University. 
The research stay was partially funded by a grant from the University of Ottawa. 

\bibliographystyle{alpha}
\bibliography{ref}

\end{document}